\documentclass[11pt,reqno,oneside]{article}
\usepackage{amsmath, amsthm, mathtools, amssymb, enumerate,esint}
\usepackage{graphicx}
\usepackage{xcolor}
\usepackage{caption}
\usepackage{datetime}
\usepackage{fancyhdr}
\usepackage{marginnote}
\usepackage[unicode,breaklinks=true,colorlinks=true,linkcolor=black,urlcolor=black,citecolor=black]{hyperref}
\usepackage{mdframed}
\usepackage{cancel}

\chardef\coloryes=1
\chardef\isitdraft=0
\chardef\forshowkeys=0
\chardef\showllabel=0
\chardef\refcheck=0
\chardef\sketches=0
\chardef\figure=0

\ifnum\forshowkeys=1
  \usepackage[notref,notcite,color]{showkeys}
\fi
\ifnum\refcheck=1
  \usepackage{refcheck}
\fi
\ifnum\isitdraft=1 
   \def\eqref#1{({\ref{#1}})}                

\numberwithin{equation}{section}

\allowdisplaybreaks[4]

\definecolor{refkey}{rgb}{.9,0.3,0.3}
\definecolor{labelkey}{rgb}{.5,0.1,0.1}
\else
  \def\startnewsection#1#2{\section{#1}\label{#2}\setcounter{equation}{0}}   
\numberwithin{equation}{section}
  \def\nnewpage{} 
\fi
\begin{document}
\def\restrx{\bigr|_{\partial\Omega}}  
\def\po{\partial\Omega}
\def\tdot{{\gdot}}
\def\intk{[k T_0, (k+1)T_0]}
\def\Dg{{D'g}}
\def\ua{u^{\alpha}}
\def\ques{{\colr \underline{??????}\colb}}
\def\nto#1{{\colC \footnote{\em \colC #1}}}
\def\fractext#1#2{{#1}/{#2}}
\def\fracsm#1#2{{\textstyle{\frac{#1}{#2}}}}   
\def\baru{U}
\def\nnonumber{}
\def\palpha{p_{\alpha}}
\def\valpha{v_{\alpha}}
\def\qalpha{q_{\alpha}}
\def\walpha{w_{\alpha}}
\def\falpha{f_{\alpha}}
\def\dalpha{d_{\alpha}}
\def\galpha{g_{\alpha}}
\def\halpha{h_{\alpha}}
\def\psialpha{\psi_{\alpha}}
\def\psibeta{\psi_{\beta}}
\def\betaalpha{\beta_{\alpha}}
\def\gammaalpha{\gamma_{\alpha}}
\def\TTalpha{T_{\alpha}}
\def\TTalphak{T_{\alpha,k}}
\def\falphak{f^{k}_{\alpha}}
\def\R{\mathbb R}
\def\norm#1{\left\Vert #1\right\Vert} 
\def\tanT{\textbf{T}}
\newcommand {\Dn}[1]{\frac{\partial #1  }{\partial N}}
\def\andand{\text{\indeq and\indeq}}
\def\mm{m}
\def\colr{{}}
\def\colg{{}}
\def\colb{{}}
\def\cole{{}}
\def\colA{{}}
\def\colB{{}}
\def\colC{{}}
\def\colD{{}}
\def\colE{{}}
\def\colF{{}}
\ifnum\coloryes=1
  \definecolor{coloraaaa}{rgb}{0.7,0.7,0.7}
  \definecolor{colorbbbb}{rgb}{0.1,0.7,0.1}
  \definecolor{colorcccc}{rgb}{0.8,0.3,0.9}
  \definecolor{colordddd}{rgb}{0.0,.5,0.0}
  \definecolor{coloreeee}{rgb}{0.8,0.3,0.9}
  \definecolor{colorffff}{rgb}{0.8,0.3,0.9}
  \definecolor{colorgggg}{rgb}{0.5,0.0,0.4}
  \definecolor{colorhhhh}{rgb}{0.7,0.7,0.7}
 \def\colg{\color{colordddd}}
 \def\coly{\color{colorhhhh}}
 \def\colb{\color{black}}
 \def\colr{\color{red}}
 \def\colu{\color{blue}}
 \def\cole{\color{colorgggg}}
  \def\cole{\color{black}}
  \def\colw{\color{coloraaaa}}
 \def\colA{\color{coloraaaa}}
 \def\colB{\color{colorbbbb}}
 \def\colC{\color{colorcccc}}
 \def\colD{\color{colordddd}}
 \def\colE{\color{coloreeee}}
 \def\colF{\color{colorffff}}
 \def\colG{\color{colorgggg}}
\fi
   \chardef\coloryes=1 
   \baselineskip=17pt
\pagestyle{myheadings}
\def\const{\mathop{\rm const}\nolimits}  
\def\HH{\text{H}}
\def\g{\geq}
\def\l{\leq}
\def\id{\mathop{\rm id}\nolimits}    
\def\diam{\mathop{\rm diam}\nolimits}    
\def\inprogress{ \text{\colC~IN~PROGRESS~}}

\ifnum\showllabel=1
 \def\llabel#1{\marginnote{\color{lightgray}\rm\small(#1)}[-0.0cm]\notag}
\else
 \def\llabel#1{\notag}
\fi

\def\rref#1{{\ref{#1}{\rm \tiny \fbox{\tiny #1}}}}
\def\theequation{\fbox{\bf \thesection.\arabic{equation}}}
\def\ccite#1{{\cite{#1}{\rm \tiny ({#1})}}}
\def\startnewsection#1#2{\newpage\colg \section{#1}\colb\label{#2}}
\setcounter{equation}{0}
\pagestyle{fancy}
\cfoot{}
\rfoot{\thepage}
\chead{}
\rhead{\thepage}
\def\nnewpage{\newpage}
\newcounter{startcurrpage}
\newcounter{currpage}
\def\llll#1{{\rm\tiny\fbox{#1}}}
   \def\blackdot{{\color{red}{\hskip-.0truecm\rule[-1mm]{4mm}{4mm}\hskip.2truecm}}\hskip-.3truecm}
   \def\bdot{{\color{blue} {\hskip-.0truecm\rule[-1mm]{4mm}{4mm}\hskip.2truecm}}\hskip-.3truecm}
   \def\purpledot{{\colA{\rule[0mm]{4mm}{4mm}}\colb}}
   \def\pdot{\purpledot}
   \def\gdot{{\colB{\rule[0mm]{4mm}{4mm}}\colb}}
\def\nts#1{{\hbox{\bf ~#1~}}} 
\def\igor#1{{\colr{\hbox{\bf IK: ~#1~}}}} 
\def\qi#1{{\hbox{\bf\color{purple} QX: ~#1~}}} 
\def\wojtek#1{{\hbox{\bf WO: ~#1~}}} 
\def\nts#1{{\colr\small\hbox{\bf ~#1~}}} 
\def\igor#1{{\colr\small\hbox{\bf IK: ~#1~}}} 
\def\ntsf#1{\footnote{\colb\hbox{\rm ~#1~}}} 
\def\bigline#1{~\\\hskip2truecm~~~~{#1}{#1}{#1}{#1}{#1}{#1}{#1}{#1}{#1}{#1}{#1}{#1}{#1}{#1}{#1}{#1}{#1}{#1}{#1}{#1}{#1}\\}
\def\biglineb{\bigline{$\downarrow\,$ $\downarrow\,$}}
\def\biglinem{\bigline{---}}
\def\biglinee{\bigline{$\uparrow\,$ $\uparrow\,$}}
\def\inon#1{\hbox{\ \ \ \ \ }\hbox{#1}}
\def\onon#1{\inon{on~$#1$}}
\def\inin#1{\inon{in~$#1$}}
\def\Omf{\Omega_{\text f}}
\def\Ome{\Omega_{\text e}}
\def\Gae{\Gamma_{\text e}}
\def\Gaf{\Gamma_{\text f}}
\def\Gac{\Gamma_{\text c}}
\def\wext{\tilde{w}}
\def\wexts{\widetilde{S w}}
\def\wexta{\overline{w}}
\def\wextb{\overline{\overline{w}}}
\def\mbar{{\overline M}}
\def\tilde{\widetilde}
\def\epsi{\epsilon_1}    
\def\epsj{\epsilon_2}     
\def\epsk{\bar\epsilon}     
\def\norm#1{\left\Vert #1\right\Vert} 
\def\nnorm#1{\Vert #1\Vert} 
\newtheorem{Theorem}{Theorem}[section]
\newtheorem{Corollary}[Theorem]{Corollary}
\newtheorem{Proposition}[Theorem]{Proposition}
\newtheorem{Lemma}[Theorem]{Lemma}
\newtheorem{Remark}[Theorem]{Remark}
\newtheorem{definition}{Definition}[section]

\def\theequation{\thesection.\arabic{equation}}
\def\endproof{\hfill$\Box$\\}
\def\square{\hfill$\Box$\\}
\def\comma{ {\rm ,\qquad{}} }            
\def\commaone{ {\rm ,\quad{}} }         
\def\dist{\mathop{\rm dist}\nolimits}    
\def\sgn{\mathop{\rm sgn\,}\nolimits}    
\def\Tr{\mathop{\rm Tr}\nolimits}    
\def\curl{\mathop{\rm curl}\nolimits}    
\def\div{\mathop{\rm div}\nolimits}    
\def\supp{\mathop{\rm supp}\nolimits}    
\def\divtwo{\mathop{{\rm div}_2\,}\nolimits}    
\def\re{\mathop{\rm {\mathbb R}e}\nolimits}    
\def\indeq{\qquad{}\!\!\!\!}                     
\def\period{.}                           
\def\semicolon{\,;}                      
\newcommand{\cD}{\mathcal{D}}
\newcommand{\eqnb}{\begin{equation}}
\newcommand{\eqne}{\end{equation}}
\newcommand{\na}{\nabla }
\newcommand{\bog}{b}
\newcommand{\bb}{\nu }
\newcommand{\ww}{{\overline{w}}}
\newcommand{\la}{\lambda }
\newcommand{\p}{\partial }
\newcommand{\N}{\mathbb{N}}
\newcommand{\T}{\mathbb{T}}
\renewcommand{\R}{\mathbb{R}}
\newcommand{\lec}{\lesssim  }
\newcommand{\gec}{\gtrsim  }
\newcommand{\tom}{\tilde{\omega}}
\newcommand{\tu}{\tilde{u}}
\newcommand{\un}{^{(n)}}
\newcommand{\unp}{^{(n+1)}}
\newcommand{\unm}{^{(n-1)}}
\newcommand{\lo}{L^2(\Omega)}
\newcommand{\lio}{L^\infty(\Omega)}
\newcommand{\ls}{L^2(S)}
\newcommand{\lis}{L^\infty(S)}

\title{Analyticity up to the boundary for the divergence equation}
\author{Igor~Kukavica and Qi Xu}

\maketitle
\date{}
\medskip
\begin{abstract}
We address analytic regularity for the divergence equation $\div u = f$ in~$\Omega$,
with $u=0$ on~$\partial\Omega$, where $\Omega$ is an arbitrary bounded analytic domain
and $\int_{\Omega} f\,dx=0$.
If $f$ is analytic on $\overline{\Omega}$, then we prove that there exists a solution that is analytic on $\overline{\Omega}$.
\end{abstract}

\noindent\thanks{\em Keywords:\/}
divergence equations, Stokes systems, analyticity regularity
\section{Introduction}
\label{Sec1}
The existence and regularity of $W_0^{k,p}$-solutions, for
$1<p<\infty$ and $k\in \mathbb{N}_0$, of the $n$-dimensional
divergence problem
\begin{align}
\begin{split}
    \div u&=f\inin{\Omega},
    \\
    u&=0 \onon{\partial\Omega}
    \label{EQ01}
\end{split}
\end{align}
is classical; see~\cite{AF,L, LS}.
Bogovski\u{\i}~\cite{B} constructed a solution operator of \eqref{EQ01} in the case where the domain is a finite union of star-shaped subdomains and also provided the corresponding regularity estimates by combining the Calderón–Zygmund theory with a partition-of-unity argument. The reader is referred to the book of Galdi~\cite{G} for a detailed exposition.
Takahashi~\cite{T} extended Bogovski\u{\i}’s formula for a bounded domain star-shaped with respect to an open ball. Furthermore, Mitrea~\cite{M} studied the case where $u$ is divergence-free and the trace of $u$ belongs to Besov or Triebel-Lizorkin spaces by treating differential forms with  Calder\'on-Zygmund theory and de~Rham theory.
Costabel and McIntosh~\cite{CM} proved the generalized Bogovski\u{\i} integral operator, acting
on differential forms in $\R^n$, is pseudodifferential operator
of order~$-1$. 
In H\"older domains, Kapitanskii and Piletskas~\cite{KP} used a dimension-reduction argument in $\R^d$, where $d\geq 2$, to obtain an explicit method of constructing solutions to~\eqref{EQ01} together with the corresponding H\"older regularity estimates. Constructive results for bounded Lipschitz domains and for $\mathbb{T}^d$ was established by Bourgain and Brezis~\cite{BB} using a nonlinear method. More recently, Chan, Chen, and Su~\cite{CCS} constructed explicit real-analytic solutions to~\eqref{EQ01} for the case of an annulus. 

In the present paper, we prove the existence of real-analytic solutions for equation~\eqref{EQ01} for any bounded analytic domain $\Omega$, as stated in our main result, stated next.

\begin{Theorem}
\label{main thm}
Assume that $f$ is a real-analytic function
on $\overline\Omega$
satisfying the compatibility
condition $\int_{\Omega}f\,d x=0$.
Then there exists
a solution to \eqref{EQ01} that is real analytic on~$\overline{\Omega}$.
\end{Theorem}

A key obstacle is the nonuniqueness for the problem~\eqref{EQ01}. The Bogovski\u{\i} operator itself relies on partition-of-unity arguments
dividing the domain into
star-shaped subdomains, which is not suited for analyticity.
Instead, we reduce the construction of solutions to~\eqref{EQ01} to the analysis of a suitable stationary Stokes system, which is discussed in detail in Section~\ref{Sec3}. For the resulting Stokes system, tangential and normal derivatives do not commute and tangential vector fields may vanish on certain subsets of the boundary. We extend the boundary derivative system by incorporating Komatsu's system~\cite{K1,K2} of tangential vector fields in Section~\ref{Sec2}. We shall apply a reduction argument for normal derivatives and introduce three Leibniz-type estimation for commutator terms in Section~\ref{Sec4}. Finally, Section~\ref{Sec5} contains the proof of the analytic regularity by choosing a suitable equivalent norm for uniformly real-analytic functions. 

\section{Analytic vector fields in a bounded domain }
\label{Sec2}
 Throughout the paper, assume that $\Omega$ is a bounded domain in $\mathbb{R}^d$, where $d\in\mathbb{N}$, with analytic boundary~$\partial \Omega$. 
We will use the notation
  \begin{equation}
    \| \cdot \|_{p} \coloneqq \| \cdot \|_{L^p (\Omega )}, \qquad \| \cdot \| \coloneqq \| \cdot \|_{2}.
   \llabel{EQ02}
   \end{equation}

\subsection{Analytic vector fields}
Denote by $\delta(x)$ the signed distance function to the boundary $\partial\Omega$,
which is positive in $\Omega$ and negative in~$\mathbb{R}^d\backslash\Omega$. For $\delta_0>0$, define
  \begin{equation}
   \Omega_{\delta_0} = \{ x \in \Omega: \delta(x) < \delta_0\}
   \andand \Omega^{\delta_0}
   = \Omega \setminus \overline{\Omega}_{\delta_0}.
   \llabel{EQ03}
  \end{equation} 
  We call a vector field $X$  \textit{tangential} along
  $\partial\Omega$ provided that $X \delta = 0$
on~$\partial\Omega$. 
Such $X$ induces an operator on functions defined on the boundary. More precisely, for any
$f\in C^\infty(\partial\Omega)$, we set $Xf := (X \tilde{f} ) \restrx$, 
where $\tilde{f} \in C^{\infty} (\overline{\Omega})$ is an arbitrary smooth extension of~$f$.

\begin{Remark}
\label{R03}
{\rm
    A function $f:\overline{\Omega}\to\mathbb{R}^k$, where $k\in\mathbb{N}$, is said to be real-analytic on $\overline \Omega$
    if there exists an open domain $\Omega'\supset\overline{\Omega}$
and a real-analytic extension $\tilde{f}\colon\Omega'\to \R^k$ such that
$    \tilde f\big|_{\overline{\Omega}}=f$.
}
\end{Remark}

The existence of global analytic vector fields stated in the next
proposition was established by Komatsu~\cite{K2}. The statement reads
as follows.

\begin{Proposition}[\text{\cite[Section~2]{K2}}]
\label{P01}
There exist analytic vector fields 
$X_0$, $T_1, \ldots, T_{N'}$, $T_{N'+1}, \ldots, T_N$ for any sufficiently small $\delta_0>0$
defined globally on $\overline{\Omega}$ satisfying the following properties: 
\begin{enumerate}
\item[1.]
$T_1, \ldots , T_N$ are \textit{tangential} to~$\partial\Omega$.
\item[2.]
On $\overline{\Omega}$,
we can represent
  \begin{equation}
   \frac{\partial}{\partial x_k} = \xi_k (x)  X_0 + \sum_{j=1}^{N} \eta_{jk} (x) T_j \comma k=1, \ldots,d, 
   \llabel{EQ05}
  \end{equation}
with coefficient $\xi_k (x)$ and $\eta_{jk}(x)$ which are analytic on~$\overline\Omega$.
\item[3.]
On $\overline{\Omega}^{\delta_0}$, we have
\begin{equation}
\frac{\partial}{\partial x_k} = \sum_{j=1}^{N'} \zeta_{jk} (x)  T_j \comma k=1,\ldots, d,
   \llabel{EQ06}
\end{equation}
where the coefficients $\zeta_{jk} (x)$ are analytic on~$\overline{\Omega}^{\delta_0}$.
\end{enumerate}
\end{Proposition}

\begin{Remark}
\label{R02}
{\rm
Denote $I = \{1, \ldots, N \}$.
We introduce the following agreement for full-space derivatives $\partial_{x}^{j}$ and tangential derivatives~$\tanT^{k}$.
The symbol $\tanT^{k}$, where $k\in{\mathbb N}$,
is interpreted in a tensorial sense, that is, it
represents the collection of all possible operators 
$T_{\beta_1} \cdots T_{\beta_k}$,
where $\beta = (\beta_1, \ldots, \beta_k) \in I^k$. We adopt a similar agreement for~$\partial_x^{j}$.
When such symbols appear inside a norm, they should be understood in the following way. We define
\begin{equation}
 \Vert\partial^{j}_x \tanT^{k} u \Vert 
 = \sum_{\substack{\alpha \in \mathbb{N}_{0}^{d} , |\alpha|=j \\ \beta \in I^k}}
 \Vert \partial^{\alpha}_x \tanT^{\beta} u \Vert,
\llabel{EQ07}
\end{equation}
where $\partial_x^{\alpha}= \partial_{x_1}^{\alpha_1}\cdots\partial_{x_d}^{\alpha_d}$ and 
$\tanT^{\beta} = T_{\beta_1} \cdots T_{\beta_k}$ 
with $\beta = (\beta_1, \ldots, \beta_k) \in I^k$, for $j \in \mathbb{N}_0$ and $k \in \mathbb{N}$.
Analogously, we use the definitions
\begin{equation}
\Vert \partial^{j}_x \tanT^{k} u \Vert_{ \dot{H}^1} = 
\Vert \partial^{j+1}_x \tanT^{k}u \Vert
\llabel{EQ08}
\end{equation}
and
\begin{equation}
\Vert \partial^{j}_x \tanT^{k} u \Vert_{ \dot{H}^2} = 
\Vert \partial^{j+2}_x \tanT^{k} u \Vert
\llabel{EQ09}
\end{equation}
}\rm to represent the homogeneous Sobolev norms.
\end{Remark}

\subsection{Equivalent norm of analytic functions}
We introduce the index sets
  \begin{equation}
   B
   =\bigl\{
     (i,j): 
    i,j\in{\mathbb N}_0,
    i+j\ge 3
    \bigr\}
    \andand
    B^{c}={\mathbb N}_{0}^{2}\backslash B,
   \llabel{EQ10}
  \end{equation}
distinguishing the two cases depending on the order.
For the divergence system \eqref{EQ01}, we define  
  \begin{align}
   \begin{split}
\rho (u) & = \sum_B \frac{1}{(i+j)!} \epsi^{i} \epsj^{j} 
 \Vert  \partial_{x}^{i} \tanT^{j}  u \Vert
  + \sum_{B^{c}} 
  \frac{1}{(i+j)!} \epsi^{i} \epsj^{j}
  \Vert \partial_{x}^{j} \tanT^{j}  u\Vert
  =
  \overline{\rho}(u) + \rho_{0} (u) \label{EQ11}
  ,
   \end{split} 
\end{align} where $\epsilon_1,\epsj$ are determined in Section~\ref{Sec5} below.
Using a standard analytic argument, we can show that the definition of uniformly real-analytic function on $\overline{\Omega}$ is equivalent to the finiteness of the norm $\rho(f)$ defined in~\eqref{EQ11}
for some $\epsilon_1,\epsilon_2>0$.

\section{Reduction of the Divergence Equation to the Stokes System}
\label{Sec3}
To address the divergence equation, we add an additional constraint;
namely, we seek $u$ which solves the
Stokes system
\begin{align}
    \begin{split}
        -\Delta u+\nabla p&=0\inin{\Omega},
        \\
        \nabla\cdot u&=f\inin{\Omega},
        \\
        u&=0 \onon{\partial\Omega}.
        \label{EQ12}
    \end{split}
\end{align}
In order to transform~\eqref{EQ12} into the divergence-free form, we
first solve the Laplace boundary value problem
\begin{align}
    \begin{split}
        -\Delta \phi&=f\inin{\Omega}',
        \\
        \phi&=0\onon{\partial\Omega}',
        \label{EQ13}
    \end{split}
\end{align}
where $\Omega'\supset \overline\Omega$ was introduced in Remark~\ref{R03}.
Setting
\begin{align}
    (v,q)=(u+\nabla\phi,p-f)
    ,
    \label{EQ14}
\end{align} 
we can easily verify that
\begin{align}
    \begin{split}
        -\Delta v+\nabla q&=0\inin{\Omega},
        \\
        \nabla\cdot v&=0\inin{\Omega},
        \\
        v&=\nabla\phi   \onon{\partial\Omega}. 
        \label{EQ15}
    \end{split}
\end{align}
Thus, we only need to prove the regularity estimate for the system~\eqref{EQ15}.
Let
\begin{align}
\begin{split}
    \psi(v,q)
    &=
    \sum_{i,j\geq0}\frac{\epsi^i\epsj^j}{(i+j)!}\|\partial_x^i\tanT^jv\|
    +\sum_{i\geq 2,j\geq0}\frac{\epsi^i\epsj^j}{(i+j)!}\|\partial_x^{i-1}\tanT^jq\|
       \\&\indeq
    +\sum_{i=1,j\geq0}\frac{\epsi\epsj^j}{(1+j)!}\|\tanT^jq\|
    +\frac{\epsj^j}{j!}\sum_{i=0,j\geq 1}\|\tanT^{j-1}q\|
    .
    \end{split}
   \llabel{EQ16}
\end{align}
According to the standard-regularity result for the Stokes system, we have
\begin{align}
    \begin{split}
        \rho_0(v)+\rho_0(q)\leq C\rho(f)
	.
    \end{split}
    \llabel{EQ17}
\end{align}
It now remains to establish the analytic regularity estimate for higher-order derivatives.

\begin{Theorem}[Stokes equations] \label{main thm_Stokes}
There exist 
$   0< \epsj\leq \epsi \leq 1$,
depending only on the dimension~$d$ and the analyticity radius  of the tangential vector field $\tanT$, such that for any $f$ satisfying the compatibility condition $\int_{\Omega}f\,d x=0$, the solution $(v,q)$ of \eqref{EQ13}--\eqref{EQ15} exists and satisfies the estimate 
\begin{align}
\rho(v)
\lec \rho (f)
,
\llabel{EQ18}
\end{align}
where $f$ is defined as~\eqref{EQ01}.
\end{Theorem}

The theorem is proven in Section~\ref{Sec5}.

\section{Derivative reductions for the Stokes system}
\label{Sec4}
In this section, we state the normal and tangential derivative reduction estimates
for a smooth solution $u$ of \eqref{EQ13}--\eqref{EQ15} in terms of the vector fields introduced in Section~\ref{Sec2}.
We also recall from \cite{CKV} necessary auxiliary statements (see
also~\cite{JKL}).

\subsection{Normal derivative reductions}
\label{sec:normal}
Here we consider the boundary value problem~\eqref{EQ15}.

\begin{Lemma}
\label{L01}
For $i \geq 2$, we have
  \begin{align}
  \begin{split}
    \Vert
     \partial_{x}^{i} \tanT^{j}  v
    \Vert
    +
    \|\partial_x^{i-1} \tanT^{j} q\|
   &  \lec
    \Vert
     \partial_{x}^{i-2}\tanT^{j+1} v
    \Vert_{H^1}  
    +
     \Vert
     \partial_{x}^{i-2}\tanT^{j} v
    \Vert
     \\&\indeq
     +
     \|\partial_x^{i-2}[\tanT^j,\Delta]v\|
     +\|\partial_x^{i-2}[\tanT^j,\nabla]q\|
    .
    \end{split}
   \label{EQ19}
  \end{align}
Similarly, for $i=1$ and $j \geq 1$, we have
\begin{align}
  \begin{split}
    \Vert
     \partial_{x} \tanT^{j} v
    \Vert
    +\|\tanT^j q\|
   &  \lec
    \Vert
    [\tanT^{j-1},\Delta] v
    \Vert
    +
    \Vert
     [\tanT^{j-1},\nabla] q
    \Vert
    +
   \Vert
     \tanT^j f
    \Vert.
  \end{split}
   \label{EQ20}
  \end{align}
\end{Lemma}

We begin by recalling the $H^2$-regularity for the stationary Stokes system
\begin{align*}
\begin{split}
    -\Delta v+\nabla q&=g\inin{\Omega},
    \\
    \nabla\cdot v&=0\inin{\Omega},
\end{split}
   \llabel{EQ21}
\end{align*}
which reads
\begin{align}
    \|v\|_{H^2}+\|\nabla q\|_{L^2}\lec\|g\|_{L^2}+\|Tv\|_{H^1}+\|v\|_{L^2}
    .
    \label{EQ22}
\end{align}

\begin{proof}[Proof of Lemma~\ref{L01}]
Using \eqref{EQ15}, we compute
\begin{align}
    \begin{split}
        -\Delta \partial_x^{i-2}\tanT^j v
        +\nabla\partial_x^{j-2}\tanT^j q
        &=
        \partial_x^{i-2}[\tanT^j,\Delta]v
        -\partial_x^{i-2}[\tanT^j,\nabla]q
        +\partial_x^{i-2}\tanT^j(-\Delta v+\nabla q)
        \\&
        =
        \partial_x^{i-2}[\tanT^j,\Delta]v
        -\partial_x^{i-2}[\tanT^j,\nabla]q
        .
        \llabel{EQ23}
    \end{split}
\end{align}
By the $H^2$-regularity estimate~\eqref{EQ22}, we get~\eqref{EQ19}.

To prove~\eqref{EQ20}, let $j\geq1$. We have
\begin{align}
    \begin{split}
        -\Delta\tanT^{j-1}v+\nabla \tanT^{j-1}q
        &=
        [\tanT^{j-1},\Delta]v
        -[\tanT^{j-1},\nabla]q
        +\tanT^{j-1}(-\Delta v+\nabla q)
        \\&=
        [\tanT^{j-1},\Delta]v
        -[\tanT^{j-1},\nabla]q
        .
        \llabel{EQ24}
    \end{split}
\end{align}
The standard $H^2$-regularity estimate for the Stokes equation leads to
\begin{align}
\begin{split}
    \|\partial_x\tanT^jv\|+\|\tanT^jq\|
    &\lec\|[\tanT^{j-1},\Delta]v\|
        +\|[\tanT^{j-1},\nabla]q\|
        +\|\tanT^{j-1}(\nabla\phi\restrx)\|_{H^\frac{3}{2}(\partial\Omega)}
    \\&
    \lec
    \|[\tanT^{j-1},\Delta]v\|
        +\|[\tanT^{j-1},\nabla]q\|
        +\|\tanT^{j}\phi\|_{H^2}
    \\&
    \lec
    \|[\tanT^{j-1},\Delta]v\|
        +\|[\tanT^{j-1},\nabla]q\|
        +\|\tanT^{j}f\|
        ,
        \label{EQ25}
    \end{split}
\end{align}
where in the last step of~\eqref{EQ25} we used the $H^2$-regularity for the Laplace equation,~\eqref{EQ13}.
\end{proof}

\subsection{Tangential derivative reduction}
The following lemma allows us 
to reduce the number of tangential derivatives.

\begin{Lemma}
\label{L02}
For $j\geq 2$  we have
\begin{align}
    \begin{split}
        \|\tanT^j v\|+\|\tanT^{j-1}q\|\lec
        \|[\tanT^{j-2},\Delta]v\|+\|[\tanT^{j-2},\nabla]q\|+\|\tanT^{j-1}f\|.
    \end{split}
    \label{EQ26}
\end{align}
\end{Lemma}

\begin{proof}[Proof of Lemma~\ref{L02}]
The statement \eqref{EQ26} follows directly from~\eqref{EQ25}.
\end{proof}

\subsection{Leibniz-type formulae}
The following three estimates can be found in \cite{CKV}, and
we reproduce them here without proof.
Part~(i) derives upper bounds for the commutators with the Laplacian,
second part examines
the operator $\partial_{x}^{i} [\tanT^{j}, \Delta]$ for $i, j \geq 1$,
while the third provides similar estimates for the pressure term.

\begin{Lemma}
\label{L03}
(i)
For $j\geq1$, we have
\begin{align}
  \begin{split}
	\Vert 
	  [\tanT^j, \Delta] v 
	\Vert
	& \lec 
	\sum_{j'= 1}^{j}  \frac{j!}{(j-j')!} K^{j'} 
	\Vert  \partial_{x}^{2} \tanT^{j-j'}  v \Vert
	+ 
	\sum_{j'= 1}^{j}  \frac{j!}{(j-j')!} K^{j'}  j' 
	\Vert \partial_{x} \tanT^{j-j'} v \Vert
     ,
  \end{split}
	\label{EQ27}
\end{align} 
for some $K>0$ depending only on the vector fields from Proposition~\ref{P01}.\\
(ii)
For  $i,j\geq1$,
 we have
 \begin{align}
\Vert 
 \partial_{x}^{i} [\tanT^{j}, \Delta] v 
\Vert
\lec
\sum_{j'=0}^{j-1} \sum_{i' + i_3 =i} \binom{i' + j-j'}{i'}\frac{i! \ j!}{i_3 ! \ j'!}  K^{i' +j-j'}
\Vert 
\partial_{x}^{i_3 +2} \tanT^{j'} v
\Vert
,
\label{EQ28}
\end{align}
for some $K >0$.\\
(iii)
For $j\geq1$, we have
\begin{align}
    \|[T^j,\nabla]q\|\lec\sum_{j'=1}^j\frac{j!}{(j-j')!}K^{j'}\|\partial_x^1\tanT^{j-j'}q\|,
    \llabel{EQ29}
\end{align} for some $K>0$.
Furthermore, for $i,j\in\mathbb{N}$,
\begin{align}
    \|\partial_x^i[\tanT^j,\nabla]q\|\lec
    \sum^{j-1}_{j'=0}\sum_{i'=0}^{i}\binom{i'+j-j'}{i'}\frac{i!j!}{(i-i')!j'!}
    K^{i'+j-j'}\|\partial_x^{i-i'+1}\tanT^{j'}q\|
    ,
    \label{EQ30}
\end{align} for some $K>0$.
\end{Lemma}

\section{Analyticity for the stationary Stokes system}
\label{Sec5}
In this section, we employ derivative-reduction arguments to control the analytic norm of $(v,q)$ by choosing a suitable pair~$(\epsi,\epsj)$. Instead of considering $\psi(v,q)$ directly, it is more convenient to classify the terms according to the number of normal derivatives.
First,
    \begin{align}
    \begin{split}
        \psi(v,q)
        &\leq \sum_{i\geq 2}\sum_{j\geq 1}\frac{1}{(i+j)!}\epsi^i\epsj^j\|\partial_x^i\tanT^jv\|
        +\sum_{i\geq 2}\sum_{j\geq 1}\frac{1}{(i+j)!}\epsi^{i-1}\epsj^j\|\partial_x^i\tanT^jq\|
        \\&\indeq+
        \sum_{j\geq 2}\frac{1}{(j+1)!}\epsi\epsj^j\|\partial_x\tanT^jv\|+
        \sum_{j\geq 2}\frac{1}{(j+1)!}\epsi\epsj^j\|\tanT^jq\|
        \\&\indeq+
        \sum_{j\geq 3}\frac{1}{j!}\epsj^j\|\tanT^jv\|+
        \sum_{j\geq 3}\frac{1}{j!}\epsj^j\|\tanT^{j-1}q\|
        +C\rho_0(f)
        \\&:=S_1+S_2+S_3
	+C\rho_0(f).
        \label{EQ31}
        \end{split}
    \end{align}
In the following subsections,
we use Lemmas~\ref{L01} and~\ref{L03}
to treat the 
first three sums in~\eqref{EQ31}.

\subsection{The sum $S_1$}\label{S_1}
Applying Lemma~\ref{L01}, we obtain
\begin{align}
    \begin{split}
        S_1&\lec
        \sum_{i\geq 2}\sum_{j\geq 1}
        \frac{\epsi^i\epsj^j}{(i+j)!}\left(
        \|\partial_x^{i-2}\tanT^{j+1}v\|_{H^1}
        +
        \|\partial_x^{i-2}\tanT^{j}v\|
        +
        \|\partial_x^{i-2}[T^j,\Delta]v\|
        +
        \|\partial_x^{i-2}[T^j,\nabla]q\|
        \right)
        \\&
        \lec
        \sum_{i\geq 2}\sum_{j\geq 1}
        \frac{\epsi^{i-1}\epsj^{j+1}}{(i+j)!}\|\partial_x^{i-1}\tanT^{j+1}v\|\frac{\epsi}{\epsj}
        +
        \sum_{i\geq 2}\sum_{j\geq 1}
        \frac{\epsi^{i-2}\epsj^{j+1}}{(i+j-1)!}\|\partial_x^{i-2}\tanT^{j+1}v\|\frac{\epsilon_1^2}{\epsj(i+j)}
        \\&\indeq
        +
        \sum_{i\geq 2}\sum_{j\geq 1}
        \frac{\epsi^{i-2}\epsj^{j}}{(i+j-2)!}\|\partial_x^{i-2}\tanT^{j}v\|\frac{\epsilon_1^2}{(i+j)(i+j-1)}
        \\&\indeq
        +
        \sum_{i\geq 2}\sum_{j\geq 1}
        \frac{\epsi^i\epsj^j}{(i+j)!}
        \left(\|\partial_x^{i-2}[T^j,\Delta]v\|
        +
        \|\partial_x^{i-2}[T^j,\nabla]q\|
        \right)
        \\&
        =:S_1^{(1)}+S_1^{(2)}+S_1^{(3)}+\text{Com}_1(v)+\text{Com}_1(q)
        .
        \label{EQ32}
    \end{split}
\end{align}
With $C$ denoting the implicit constant in the resulting inequality of \eqref{EQ32},
we choose
$\epsilon_1$ and $\epsilon_2$ such that
\begin{align}
    \begin{split}
        \epsi = \frac{1}{100(C+1)(K+1)}\epsj\leq \frac{1}{10^6(C+1)^4(K+1)^4}.
        \llabel{EQ33}
    \end{split}
\end{align}
Therefore, 
\begin{align}
    \begin{split}
        S_1^{(1)}+S_1^{(2)}+S_1^{(3)}\leq \frac{1}{20}\psi(v,q).
        \label{EQ34}
    \end{split}
\end{align}
By Lemma~\ref{L03}(i),(ii) and Fubini's theorem, we derive
\begin{align}
    \begin{split}
     &\sum_{i\geq 2}\sum_{j\geq 1}
        \frac{\epsi^i\epsj^j}{(i+j)!}
        \|\partial_x^{i-2}[T^j,\Delta]v\|
        =\sum_{i\geq 0}\sum_{j\geq 1}
        \frac{\epsi^{i+2}\epsj^j}{(i+j+2)!}
        \|\partial_x^{i}[T^j,\Delta]v\|
    \\&\indeq
    \lec
    \sum_{i\geq 0}\sum_{j\geq 1}
    \sum_{j'=0}^{j-1} \sum_{i' + i_3 =i} \frac{\epsi^{i+2}\epsj^j}{(i+j+2)!}
    \binom{i' + j-j'}{i'}\frac{i! \ j!}{i_3 ! \ j'!}  K^{i' +j-j'}
    \Vert 
    \partial_{x}^{i_3 +2} \tanT^{j'} v 
    \Vert
    \\&\indeq
    =\sum_{j'=0}^\infty\sum_{i_3=0}^\infty
    \left(
    \sum_{i=i_3}^\infty\sum_{j=j'+1}^\infty\binom{i-i_3+j-j'}{i-i_3}\frac{i!j!}{i_3!j'!}K^{i-i_3+j-j'}
    \frac{(i_3+j'+2)!}{(i+j+2)!}\epsi^{i-i_3}\epsj^{j-j'}
    \right)
    \\&\indeq\indeq
    \times
    \epsi^{i_3+2}\epsj^{j'}\frac{1}{(i_3+j'+2)!}\|\partial_x^{i_3+2}\tanT^{j'}v\|,
    \end{split}
   \llabel{EQ35}
  \end{align}
and then using
\begin{align}
    \begin{split}
        &\sum_{i=i_3}^\infty\sum_{j=j'+1}^\infty
        \binom{i-i_3+j-j'}{i-i_3}\binom{i_3+j'}{i_3}\binom{i+j}{i}^{-1}
        \\&\indeq\indeq\indeq\indeq\indeq\indeq\indeq\times
        \frac{(i_3+j'+1)(i_3+j'+2)}{(i+j+1)(i+j+2)}(K\epsi)^{i-i_3}(K\epsj)^{j-j'}
        \\&\indeq
	\leq
        \sum_{i=i_3}^\infty\sum_{j=j'+1}^\infty
        (K\epsi)^{i-i_3}(K\epsj)^{j-j'}
        \\&\indeq
	\leq
	\frac{1}{50(C+1)}.
    \end{split}
   \llabel{EQ36}
\end{align}
By the Vandermonde identity, we obtain
\begin{align}
    \begin{split}
        \text{Com}_1(v)\leq \frac{1}{50}\psi(v,q).
        \label{EQ37}
    \end{split}
\end{align}
We treat the commutator involving pressure similarly to~\eqref{EQ30} as
\begin{align}
    \begin{split}
        &\sum_{i\geq 2}\sum_{j\geq 1}\frac{\epsi^i\epsj^j}{(i+j)!}\|\partial_x^{i-2}[T^j,\nabla]q\|
        \\&\indeq
	\lec
        \sum_{i\geq 2}\sum_{j\geq 1}\sum_{j'=0}^{j-1}\sum_{i'=0}^{i-2}
        \binom{i'+j-j'}{i'}\frac{(i-2)!j!}{(i-2-i')!j'!}K^{i'+j-j'}
        \|\partial_x^{i-1-i'}\tanT^{j'}p\|\frac{\epsi^i\epsj^j}{(i+j)!}
        \\&\indeq
	\lec
        \sum_{j'=0}^\infty\sum_{i_3=2}^\infty\sum_{j=j'+1}^\infty\sum_{i=i_3}^\infty
        \underbrace{\binom{i-i_3+j-j'}{i-i_3}\binom{i_3+j'}{i_3}\binom{i+j}{i}^{-1}
        \frac{(i_3-1)i_3}{(i-1)i}}_{\leq 1}
        \\&\indeq\indeq\indeq\indeq\indeq\indeq\indeq
	\times
        (K\epsi)^{i-i_3}(K\epsj)^{j-j'}
        \frac{\epsi^{i_3}\epsj^{j'}}{(i_3+j')!}\|\partial_x^{i_3-1}\tanT^{j'}q\|
        .
        \end{split}
   \llabel{EQ38}
\end{align}
Therefore, we obtain
\begin{align}
    \text{Com}_1(q)\leq\frac{1}{50}\psi(v,q).
    \label{EQ39}
\end{align}
Combining~\eqref{EQ34},~\eqref{EQ37}, and~\eqref{EQ39}, we get
\begin{align}
    S_1\leq \frac{1}{10}\psi(v,q),
    \label{EQ40}
\end{align}
which can be absorbed by the left-hand side of~\eqref{EQ31}.

\subsection{The $S_2$ term}\label{S_2}
Using the normal derivative reductions from Lemma~\ref{L01}, we write
\begin{align}
    \begin{split}
        S_2&\lec\sum_{j\geq 2}\frac{\epsi\epsj^j}{(j+1)!}
        \left(
        \Vert
    [\tanT^{j-1},\Delta] v
    \Vert
    +
    \Vert
     [\tanT^{j-1},\nabla] q
    \Vert
    +
   \Vert
     \tanT^j f
    \Vert
    \right)
    \\&:=\text{Com}_2(v)+\text{Com}_2(q)+\rho(f).
    \label{EQ41}
    \end{split}
\end{align}
By Lemma~\ref{L03}(i),
we have
\begin{align}
  \begin{split}
  \text{Com}_2(v)
	& \lec 
	\sum_{j\geq 2}\sum_{j'= 1}^{j-1}  \frac{(j-1)!}{(j-j'-1)!} K^{j'} 
	\Vert  \partial_{x}^{2} \tanT^{j-j'-1}  v \Vert\frac{\epsi\epsj^j}{(1+j)!}
	\\&\indeq
    + 
	\sum_{j\geq 2}\sum_{j'= 1}^{j-1}  \frac{(j-1)!}{(j-j'-1)!} K^{j'}  j' 
	\Vert \partial_{x} \tanT^{j-j'-1} v \Vert\frac{\epsi\epsj^j}{(1+j)!}.
  \end{split}
	\llabel{EQ42}
\end{align} 
Let $j_3=j-j'-1$. Using Fubini's theorem again,
\begin{align}
    \begin{split}
        \text{Com}_2(v)&\lec\sum_{j'\geq 1}\sum_{j_3=0}^\infty(K\epsj)^{j'}\frac{\epsj}{\epsi}
        \frac{(j_3+1)(j_3+2)}{(j'+j_3+1)(j'+j_3+2)}
        \frac{\epsi^2\epsj^{j_3}}{(j_3+2)!}\|\partial_x^2\tanT^{j_3}u\|
        \\&\indeq+\sum_{j'\geq1}\sum_{j_3=0}^\infty(K\epsj)^{j'}\frac{j'(j_3+1)\epsj}{(j'+j_3+1)(j'+j_3+2)}
        \frac{\epsi\epsj^{j_3}}{(j_3+1)!}\|\partial_x\tanT^{j_3}u\|
        \\&\lec
        \underbrace{\frac{\epsj}{\epsi}\sum_{j'\geq 1}(K\epsj)^{j'}}_{\leq 2K\epsj^2{\epsi}^{-1}}\sum_{j_3=0}^\infty\frac{\epsi^2\epsj^{j_3}}{(j_3+2)!}\|\partial_x^2\tanT^{j_3}u\|
        \\&\indeq+\underbrace{\epsj
        \sum_{j'\geq 1}(K\epsj)^{j'}}_{\leq2K\epsj^2}\sum_{j_3=0}^\infty\frac{\epsi\epsj^{j_3}}{(j_3+1)!}\|\partial_x\tanT^{j_3}u\|
        \\&
	\leq\frac{1}{50}\psi(v,q).
        \label{EQ43}
    \end{split}
\end{align}
To estimate the commutator involving the pressure, we use Lemma~\ref{L03}(iii) and Fubini's theorem to obtain
\begin{align}
    \begin{split}
        \text{Com}_2(q)&\lec
        \sum_{j\geq 2}\sum_{j'=1}^{j-1}\frac{(j-1)!}{(j-j'-1)!}K^{j'}\|\partial_x^1\tanT^{j-j'-1}q\|
        \frac{\epsi\epsj^j}{(1+j)!}
        \\&\lec
        \sum_{j_3=0}^\infty
        \underbrace{\sum_{j'=1}^\infty\frac{\epsj(j_3+1)(j_3+2)(K\epsj)^{j'}}{\epsi(j_3+j'+1)(j_3+j'+2)}}_{\leq K\epsj^2\epsi^{-1}}
        \frac{\epsi^2\epsj^{j_3}}{(j_3+2)!}\|\partial_x\tanT^{j_3}p\|
        \leq\frac{1}{50}\psi(v,q).
        \label{EQ44}
    \end{split}
\end{align}
Combining~\eqref{EQ41},~\eqref{EQ43}, and~\eqref{EQ44}, we get
\begin{align}
    S_2\leq \frac{1}{25}\psi(v,q)+\rho(f),
    \label{EQ45}
\end{align}
where $\frac{1}{25}\psi(v,q)$ can also be absorbed by the left-hand side of~\eqref{EQ31}.

\subsection{The $S_3$ term}\label{S_3}
Finally, for $S_3$ we use Lemma~\ref{L03}(i) and obtain
\begin{align}
    \begin{split}
        S_3&\lec\sum_{j\geq 3}\frac{\epsj^j}{j!}
        \left(
        \Vert
    [\tanT^{j-2},\Delta] v
    \Vert
    +
    \Vert
     [\tanT^{j-2},\nabla] q
    \Vert
    +
   \Vert
     \tanT^{j-1} f
    \Vert
    \right)
    \\&:=\text{Com}_3(v)+\text{Com}_3(q)+\rho(f).
    \label{EQ46}
    \end{split}
\end{align}
Again, by Lemma~\ref{L03}(i), we derive
\begin{align}
    \begin{split}
        \text{Com}_3(v)& \lec 
	\sum_{j\geq 3}\sum_{j'= 1}^{j-2}  \frac{(j-2)!}{(j-j'-2)!} K^{j'} 
	\Vert  \partial_{x}^{2} \tanT^{j-j'-2}  v \Vert\frac{\epsj^j}{j!}
	\\&\indeq
    + 
	\sum_{j\geq 3}\sum_{j'= 1}^{j-2}  \frac{(j-2)!}{(j-j'-2)!} K^{j'}  j' 
	\Vert \partial_{x} \tanT^{j-j'-2} v \Vert\frac{\epsj^j}{j!}.
    \end{split}
   \llabel{EQ47}
\end{align}
Let $j_4=j-j'-2$. Fubini's theorem implies
\begin{align}
    \begin{split}
        \text{Com}_3(v)
        &\lec\sum_{j'\geq 1}\sum_{j_4=0}^\infty(K\epsj)^{j'}
        \frac{\epsj^2}{\epsi^2}
        \frac{(j_4+1)(j_4+2)}{(j'+j_4+1)(j'+j_4+2)}
        \frac{\epsi^2\epsj^{j_4}}{(j_4+2)!}\|\partial_x^2\tanT^{j_4}v\|
        \\&\indeq+\sum_{j'\geq1}\sum_{j_4=0}^\infty(K\epsj)^{j'}
        \frac{\epsj^2}{\epsi}
        \frac{j'(j_4+1)}{(j'+j_4+1)(j'+j_4+2)}
        \frac{\epsi\epsj^{j_4}}{(j_4+1)!}
        \|\partial_x\tanT^{j_4}v\|
        \\&\lec
        \underbrace{
        \frac{\epsj^2}{\epsi^2}\sum_{j'\geq 1}(K\epsj)^{j'}}_{2K\epsj^3\epsi^{-2}}
        \sum_{j_4=0}^\infty
        \frac{\epsi^2\epsj^{j_4}}{(j_4+2)!}\|\partial_x^2\tanT^{j_4}v\|
        \\&\indeq
        +
        \underbrace{
        \frac{\epsj^2}{\epsi}\sum_{j'\geq 1}(K\epsj)^{j'}}
        _{2K\epsj^3\epsi^{-1}}
        \sum_{j_4=0}^\infty\frac{\epsi\epsj^{j_4}}{(j_4+1)!}\|\partial_x\tanT^{j_4}v\|
        \\&\leq\frac{1}{50}\psi(v,q).
        \label{EQ48}
    \end{split}
\end{align}
Similarly, we use Lemma~\ref{L03}(iii) and Fubini's theorem to obtain
\begin{align}
    \begin{split}
        \text{Com}_3(q)&\lec
        \sum_{j\geq 3}\sum_{j'=1}^{j-2}\frac{(j-2)!}{(j-j'-2)!}K^{j'}\|\partial_x^1\tanT^{j-j'-2}q\|
        \frac{\epsj^j}{j!}
        \\&\lec
        \sum_{j_4=0}^\infty
        \underbrace{\sum_{j'=1}^\infty\frac{\epsj^2(j_4+1)(j_4+2)(K\epsj)^{j'}}{\epsi^2(j_4+j'+1)(j_4+j'+2)}}_{\leq K\epsj^3\epsi^{-2}}
        \frac{\epsi^2\epsj^{j_4}}{(j_4+2)!}\|\partial_x\tanT^{j_4}q\|
        \\&
        \leq\frac{1}{50}\psi(v,q).
        \label{EQ49}
    \end{split}
\end{align}
Combining~\eqref{EQ46},~\eqref{EQ48}, and~\eqref{EQ49}, we get
\begin{align}
    S_3\leq \frac{1}{25}\psi(v,q)+\rho(f),
    \label{EQ50}
\end{align}
where $\frac{1}{25}\psi(v,q)$ can also be absorbed by the left-hand side of~\eqref{EQ31}.

\subsection{Conclusion of the proof}\label{subsection: conclusion}
\begin{proof}[Proof of Theorems~\ref{main thm} and~\ref{main thm_Stokes}]
Combining all the estimates \eqref{EQ31}, \eqref{EQ40}, \eqref{EQ45}, and \eqref{EQ50}, we write
\begin{align}
    \begin{split}
        \psi(v,q)\leq\frac{1}{5}\psi(v,q)+2\rho(f)+C\rho_0(f).
        \llabel{EQ51}
    \end{split}
\end{align}
Therefore, we have
\begin{align}
    \rho(v)\leq\psi(v,q)\lec\rho(f),
   \llabel{EQ52}
\end{align}
which concludes the proof of Theorem~\ref{main thm_Stokes}.
Recall that $u=v-\nabla\phi$, where $\phi$ is defined by~\eqref{EQ13}. Using the harmonic estimate \begin{align}
    \rho(\nabla\phi)\lec \rho(f),
    \llabel{EQ53}
\end{align}
we obtain
\begin{align}
    \begin{split}
        \rho(u)\leq\psi(v,q)+\rho(\nabla\phi)\lec \rho(f)
	,
        \llabel{EQ54}
    \end{split}
\end{align}
and the proof of Theorem~\ref{main thm_Stokes}
(and thus also of Theorem~\ref{main thm})
is concluded.
\end{proof}

\section*{Acknowledgments}
The authors were supported in part by the NSF grant DMS-2205493.

\end{document}